\documentclass[12pt]{amsart}

\usepackage{amssymb}

\usepackage{enumerate}

\usepackage{graphicx}

\makeatletter
\@namedef{subjclassname@2010}{%
  \textup{2010} Mathematics Subject Classification}
\makeatother

\newtheorem{thm}{Theorem}[section]
\newtheorem{cor}[thm]{Corollary}
\newtheorem{lem}[thm]{Lemma}

\theoremstyle{definition}

\numberwithin{equation}{section}

\begin{document}


\baselineskip=17pt



\title[On uniform approximation to real numbers]{On uniform approximation to real numbers}

\author[Y. Bugeaud]{Yann Bugeaud}
\address{IRMA, U.M.R. 7501, Universit\'e de Strasbourg et CNRS, 
7 rue Ren\'e Descartes,  67084 Strasbourg, France}
\email{bugeaud@math.unistra.fr}

\author[J. Schleischitz]{Johannes Schleischitz}
\address{Institute of Mathematics, Department of Integrative Biology, BOKU Wien, 1180, Vienna, Austria.}
\email{johannes.schleischitz@boku.ac.at}

\date{}

\begin{abstract}
Let $n \ge 2$ be an integer and $\xi$ a transcendental real number. 
We establish several new relations between the values 
at $\xi$ of the exponents of Diophantine approximation 
$w_n, w_{n}^{\ast}, \widehat{w}_{n}$, and $\widehat{w}_{n}^{\ast}$. 
Combining our results with recent estimates by Schmidt and Summerer
allows us to refine the inequality $\widehat{w}_{n}(\xi) \le 2n-1$ 
proved by Davenport and Schmidt in 1969. 
\end{abstract}



\maketitle

\section{Introduction} \label{sektion1}

Throughout the present paper, the height $H(P)$ 
of a complex polynomial $P(X)$ is the maximum of the moduli of its
coefficients and the height $H(\alpha)$ of an algebraic number
$\alpha$ is the height of its minimal polynomial over $\mathbb{Z}$. 
For an integer $n \ge 1$, the exponents of Diophantine approximation 
$w_n, w_{n}^{\ast}, \widehat{w}_{n}$, and $\widehat{w}_{n}^{\ast}$ 
measure the quality of approximation to real numbers by algebraic
numbers of degree at most $n$. 
They are defined as follows. 

Let $\xi$ be a real number.
We denote by $w_{n}(\xi)$ the supremum of the real numbers $w$ for which
$$
0<\vert P(\xi)\vert\leq H(P)^{-w}
$$
has infinitely many solutions in polynomials $P$ in ${\mathbb{Z}[X]}$ of degree at most $n$,  
and by $\widehat{w}_{n}(\xi)$ the supremum of the real numbers $w$ for which the system
\[
0<\vert P(\xi)\vert\leq H^{-w}, \qquad H(P)\leq H,
\]
has a solution $P$ in ${\mathbb{Z}[X]}$ of degree at most $n$, for all large values of $H$.

Likewise, we denote by $w_{n}^{\ast}(\xi)$ 
the supremum of the real numbers $w$ for which
$$
0<\vert \xi-\alpha\vert \leq H(\alpha)^{-w-1}
$$
has infinitely many solutions in algebraic numbers $\alpha$ of degree at most $n$, and by 
$\widehat{w}_{n}^{\ast}(\xi)$ the supremum of the real numbers $w$ for which the system
$$
0<\vert \xi-\alpha\vert \leq H(\alpha)^{-1} H^{-w}, \qquad H(\alpha)\leq H,
$$
is satisfied by an algebraic number $\alpha$ of degree at most $n$, for all large values of $H$. 

It is easy to check that every real number $\xi$ satisfies     
$$
w_{1}(\xi)=w_{1}^{\ast}(\xi) 
\quad \hbox{and} \quad 
\widehat{w}_{1}(\xi)=\widehat{w}_{1}^{\ast}(\xi). 
$$
Furthermore, if $n$ is a positive integer and $\xi$ a real number which is not algebraic  
of degree at most $n$, then Dirichlet's Theorem implies that
\begin{equation} \label{eq:derdirichlet}
w_{n}(\xi)\geq \widehat{w}_{n}(\xi)\geq n. 
\end{equation}
By combining \eqref{eq:derdirichlet} with the Schmidt Subspace Theorem, 
we can deduce that, for all positive integers $d, n$,
every real algebraic number $\xi$ of degree $d$ satisfies
\[
w_{n}(\xi)=\widehat{w}_{n}(\xi)=w_{n}^{\ast}(\xi)=\widehat{w}_{n}^{\ast}(\xi)=\min\{n,d-1\},
\]
see~\cite[Theorem~2.4]{buglaur}. Thus, we may restrict our attention 
to transcendental real numbers and, in the sequel, $\xi$ will always denote a transcendental
real number. Furthermore, in the sense of Lebesgue measure, almost 
all real numbers $\xi$ satisfy
\[
w_{n}(\xi)=\widehat{w}_{n}(\xi)=w_{n}^{\ast}(\xi)=\widehat{w}_{n}^{\ast}(\xi)=n, \quad
\hbox{for $n \ge 1$}. 
\]
The survey~\cite{bdraft} gathers the known results on the exponents 
$w_{n}^{\ast},\widehat{w}_{n}^{\ast},w_{n},\widehat{w}_{n}$, along with some 
open questions; see also \cite{bugbuch,wald}. 

A central open problem, often referred to 
as the Wirsing conjecture \cite{wirsing,bugbuch}, 
asks whether every transcendental real number $\xi$ satisfies 
$w_{n}^{\ast}(\xi)\geq n$ for every integer $n \ge 2$. It has been 
solved by Davenport and Schmidt \cite{davsh67} for $n=2$ (see also \cite{Mosh14}), but 
remains wide open for $n \ge 3$. 
In this direction, Bernik and Tishchenko~\cite{bertish} established that
\begin{equation} \label{eq:bertis}
w_{n}^{\ast} (\xi) \ge \frac{n+\sqrt{n^{2}+16n-8}}{4}
\end{equation}
holds for every integer $n \ge 3$ and every transcendental real number $\xi$. 
The lower bound \eqref{eq:bertis} was subsequently slightly refined   
by Tsishchanka~\cite{Tsi07}; 
see \cite{bugbuch} for additional references. 

Among the known relations between the exponents 
$w_{n}^{\ast},\widehat{w}_{n}^{\ast},w_{n},\widehat{w}_{n}$, let us mention that
Schmidt and Summerer~\cite[(15.4')]{ssch} used their deep, new theory 
of parametric geometry of numbers to establish that 
\begin{equation} \label{eq:ssmj}
w_{n}(\xi)\geq (n-1)
\frac{\widehat{w}_{n}(\xi)^{2}-\widehat{w}_{n}(\xi)}{1+(n-2)\widehat{w}_{n}(\xi)}
\end{equation}
holds for $n\geq 2$ and every transcendental real number $\xi$. 
This extends an earlier result of Jarn\'\i k \cite{Jar54} which 
deals with the case $n=2$.
For $n=3$ Schmidt and Summerer~\cite{sums} established the better bound 
\begin{equation} \label{eq:beesser}
w_{3}(\xi)\geq \frac{\widehat{w}_{3}(\xi)\cdot (\sqrt{4 \widehat{w}_{3}(\xi)-3}-1)}{2}.
\end{equation}

In 1969, Davenport and Schmidt \cite{davsh} proved that 
every transcendental real number $\xi$ satisfies
\begin{equation} \label{eq:glmschr}
1\leq \widehat{w}_{n}^{\ast}(\xi)\leq \widehat{w}_{n}(\xi)\leq 2n-1,  
\end{equation}
for every integer $n \ge 1$ (the case $n=1$ is due to Khintchine \cite{Kh26}). 
The stronger inequality
\begin{equation} \label{eq:glmschr2}
\widehat{w}_{2}(\xi) \leq \frac{3 + \sqrt{5}}{2} 
\end{equation}
was proved by Arbour and Roy~\cite{arbroy}; it can also be obtained by a direct
combination of another result of \cite{davsh} with a transference theorem of 
Jarn\'\i k \cite{jarnik}, which remained forgotten until 2004.  
The first inequality in \eqref{eq:glmschr} is sharp for every $n \ge 1$; 
see~\cite[Proposition~2.1]{buglaur}. 
Inequality \eqref{eq:glmschr2} is also sharp: Roy \cite{royfr,daroy} 
proved the existence of transcendental 
real numbers $\xi$ for which $\widehat{w}_{2}(\xi) = \frac{3 + \sqrt{5}}{2}$ 
and called them {\it extremal numbers}. We also point out the relations
\begin{equation} \label{eq:sternhut}
w_{n}^{\ast}(\xi)\leq w_{n}(\xi)\leq w_{n}^{\ast}(\xi)+n-1, \quad 
\widehat{w}_{n}^{\ast}(\xi)\leq \widehat{w}_{n}(\xi)\leq \widehat{w}_{n}^{\ast}(\xi)+n-1,
\end{equation}
valid for every integer $n\geq 1$ and every transcendental  
real number $\xi$, see~\cite[Lemma~A.8]{bugbuch} 
or \cite[Theorem 2.3.1]{bdraft}.  

In view of the lower bound
\begin{equation} \label{eq:wirrwarr}
w_{n}^{\ast}(\xi)\geq \frac{\widehat{w}_{n}(\xi)}{\widehat{w}_{n}(\xi)-n+1},
\end{equation}
established in~\cite{buglaur} and valid for every integer $n \ge 2$
and every real transcendental number $\xi$, 
any counterexample $\xi$ to the Wirsing conjecture 
must satisfy $\widehat{w}_{n}(\xi) > n$ for some
integer $n \ge 3$. It is unclear whether transcendental real numbers
with the latter property do exist. The main purpose of the present paper is to obtain
new upper bounds for $\widehat{w}_{n}(\xi)$ and, in particular, to improve the last
inequality of \eqref{eq:glmschr} for every integer $n \ge 3$.

\section{Main results} 

Our main result is the following improvement of the upper bound
\eqref{eq:glmschr} of Davenport and Schmidt \cite{davsh}.

\begin{thm} \label{giltgleichheit}
Let $n\geq 2$ be an integer and $\xi$ a real transcendental number. Then 
\begin{equation} \label{eq:tomcat1}
\widehat{w}_{n}(\xi)\leq   n - \frac{1}{2} + \sqrt{n^{2}-2n+ \frac{5}{4}}. 
\end{equation}
For $n=3$ we have the stronger estimate
\begin{equation} \label{eq:tomcat2}
\widehat{w}_{3} (\xi)\leq 3+\sqrt{2} = 4.4142 \ldots 
\end{equation}
\end{thm}

For $n=2$, Theorem \ref{giltgleichheit} provides an alternative proof 
of \eqref{eq:glmschr2}. This inequality  
is best possible, as already mentioned in the Introduction. 

For $n \ge 3$, Theorem \ref{giltgleichheit} gives the first improvement on \eqref{eq:glmschr}. 
This is, admittedly, a small improvement, since
for $n\ge 4$ the right hand side of \eqref{eq:tomcat1} can be 
written $2n- \frac{3}{2}+ \varepsilon_n$, 
where $\varepsilon_n$ is positive and $\lim_{n \to + \infty} \varepsilon_n = 0$.  
There is no reason to believe that our bound is best 
possible for $n \ge 3$. 

Theorem \ref{giltgleichheit} follows from the next two statements 
combined with the lower bounds \eqref{eq:ssmj} and \eqref{eq:beesser}
of $w_n(\xi)$ in terms of 
$\widehat{w}_{n}(\xi)$ obtained by Schmidt and Summerer~\cite{ssch,sums}. 

\begin{thm} \label{zendent}
Let $m \ge n \ge 2$ be integers and $\xi$ a transcendental real number.
Then {\upshape(}at least{\upshape)} one of the two assertions 
\begin{equation} \label{eq:synode}
w_{n-1}(\xi)=w_{n}(\xi)=w_{n+1}(\xi)=\cdots=w_{m}(\xi), 
\end{equation}
or
\begin{equation} \label{eq:denngiltsea} 
\widehat{w}_{n}(\xi)\leq m+(n-1)\frac{\widehat{w}_{n}(\xi)}{w_{m}(\xi)}, 
\end{equation}
holds. In other words, the inequality $w_{n-1}(\xi)<w_{m}(\xi)$ implies \eqref{eq:denngiltsea}.
\end{thm}

We remark that $w_{m}(\xi)$ may be infinite in Theorem \ref{zendent}, and this is
also the case in Theorems \ref{zeitung} and~\ref{simulterne}. 
By \eqref{eq:glmschr}, 
the inequality \eqref{eq:denngiltsea} always holds for $m \ge 2n-1$, thus 
Theorem \ref{zendent} is of interest only for $n \le m \le 2n-2$. 

For our main result Theorem~\ref{giltgleichheit} we only need the case $m=n$ of Theorem~\ref{zendent}.
We believe that at least in this case the assumption $w_{n-1}(\xi)<w_{m}(\xi)$ 
for \eqref{eq:denngiltsea} can be removed. This is indeed the case if
$\widehat{w}_{n}(\xi) = \widehat{w}_{n}^{\ast}(\xi)$; see 
Theorem \ref{simulterne}. 

\begin{thm} \label{zeitung}
Let $m,n$ be positive integers and $\xi$ be a transcendental real number. Then 
\[
\min\{w_{m}(\xi),\widehat{w}_{n}(\xi)\}\leq m+n-1.
\]
\end{thm} 

Taking $m=n$ in Theorem \ref{zeitung} gives \eqref{eq:glmschr}, but our proof
differs from that of Davenport and Schmidt. 
The choice $m=1$ in Theorem \ref{zeitung} yields the main claim of 
\cite[Theorem~5.1]{j3}, which asserts that every real number $\xi$
with $w_{1}(\xi)\geq n$ satisfies 
$\widehat{w}_{j}(\xi)=j$ for $1\leq j\leq n$. 
Theorem~\ref{zeitung} gives new information for $2 \le m \le n-1$.

A slight modification of the proof of Theorem \ref{zendent} gives the next result. 

\begin{thm} \label{simulterne}
Let $m, n$ be positive integers and $\xi$ a transcendental real number. 
Assume that either $m\geq n$ or
\begin{equation} \label{eq:negate}
w_{m}(\xi) > \min\{n+m-1,w_{n}^{\ast}(\xi)\} 
\end{equation}
is satisfied. Then
\begin{equation} \label{eq:value}
\widehat{w}_{n}^{\ast}(\xi)\leq 
\min\left\{ m+(n-1)\frac{\widehat{w}_{n}^{\ast}(\xi)}{w_{m}(\xi)}, w_{m}(\xi)\right\}.
\end{equation}
In particular, for any integer $n\geq 1$ and any transcendental real $\xi$ we have
\[
\widehat{w}_{n}^{\ast}(\xi)\leq n+(n-1)\frac{\widehat{w}_{n}^{\ast}(\xi)}{w_{n}(\xi)}. 
\]
\end{thm}

By \eqref{eq:glmschr}, 
the inequality \eqref{eq:value} always holds for $m \ge 2n-1$, thus 
Theorem \ref{simulterne} is of interest only for $1 \le m \le 2n-2$.

Let $m \ge 2$ be an integer. According to LeVeque \cite{leveque}, a real
number $\xi$ is a $U_m$-number if $w_m (\xi)$ is infinite and
$w_{m-1} (\xi)$ is finite. Furthermore, the $U_1$-numbers are
precisely the Liouville numbers, that is, the real  
numbers for which the inequalities $0 < \vert \xi-p/q\vert < q^{-w}$ have infinitely  
many rational solutions $p/q$ for every real number $w$.     
A $T$-number is a real number $\xi$ such that $w_n (\xi)$ is finite 
for every integer $n$ and $\limsup_{n \to + \infty} \, \frac{w_n (\xi)}{n} = + \infty $. 
LeVeque \cite{leveque} proved the 
existence of $U_m$-numbers for every positive integer $m$. 
Schmidt~\cite{tnumbers} was the first to confirm that $T$-numbers do exist. 
Additional results on $U_m$- and $T$-numbers and on Mahler's 
classification of real numbers are given in \cite{bugbuch}. 
The next statement is an easy consequence of our theorems. 

\begin{cor} \label{mahlerklasse}
Let $m$ be a positive integer. Every $U_m$-number $\xi$ 
satisfies $\widehat{w}_{m} (\xi) = m$
and the inequalities $\widehat{w}_{n}^{\ast}(\xi)\leq m$ and $\widehat{w}_{n}(\xi)\leq m+n-1$   
for every integer $n\geq 1$. 
Moreover, every $T$-number $\xi$ satisfies 
$\liminf_{n\to +\infty}  \frac{\widehat{w}_{n}(\xi)}{n} =1$.
\end{cor}

\begin{proof}
Let $m$ be a positive integer and $\xi$ a $U_m$-number. 
We have already mentioned that $\widehat{w}_{1} (\xi) = 1$. 
For $m \ge 2$, we have $w_{m-1} (\xi) < w_m (\xi)$ and 
we get from Theorem~\ref{zendent} that $\widehat{w}_{m} (\xi) = m$.
The bound for $\widehat{w}_{n}^{\ast}(\xi)$ follows from \eqref{eq:value} 
as we check the conditions are satisfied in both cases $m\geq n$ and $n < m$.   
from the inequalities $\widehat{w}_{n}^{\ast}(\xi) \le \widehat{w}_{m}^{\ast}(\xi)\le \widehat{w}_{m}(\xi)$.  
The upper bound  
$\widehat{w}_{n}(\xi)\leq m+n-1$ is then a consequence of \eqref{eq:sternhut}.
Let $\xi$ be a $T$-number. Then, for any positive real number $C$, there are
arbitrarily large integers $n$ such that $w_{n}(\xi)>w_{n-1}(\xi)$ and 
$w_{n}(\xi)\geq C n$. For such an $n$,
inserting these relations in \eqref{eq:denngiltsea} with $m=n$ and using \eqref{eq:glmschr}, we obtain
\[
\widehat{w}_{n}(\xi)\leq n+\frac{(n-1)(2n-1)}{Cn}<n\cdot\left(1+\frac{2}{C}\right).
\]
It is then sufficient to let $C$ tend to infinity.   
\end{proof}

Roy \cite{daroy} proved that every extremal number $\xi$ satisfies 
\begin{equation} \label{eq:gliech}
w_{2}(\xi) =\sqrt{5}+2 = 4.2361 \ldots = (\widehat{w}_{2}(\xi) - 1) \widehat{w}_{2}(\xi),
\end{equation}
thus provides a non-trivial example that equality
can hold in \eqref{eq:ssmj}.  
Approximation to extremal numbers by algebraic numbers of bounded degree
was studied in \cite{adbu,Royzelo}. We deduce from 
Theorems \ref{simulterne} and \ref{zeitung} some additional information. 

\begin{cor} \label{sternschr}
Every extremal number $\xi$ satisfies 
$$
\widehat{w}_{3}^{\ast}(\xi)\leq 3 \, \frac{2+\sqrt{5}}{1+\sqrt{5}} = 3.9270\ldots 
\quad \hbox{and} \quad 
\widehat{w}_{3}(\xi)\leq 4.
$$
\end{cor}

\begin{proof}
Let $m=2, n=3$ and $\xi$ be an extremal number.
By \eqref{eq:gliech} we have $w_{2}(\xi)=2+\sqrt{5}>4=m+n-1$ and the 
first claim follows from \eqref{eq:value}. 
Theorem~\ref{zeitung} implies the second assertion.
\end{proof}

We conclude this section by a new relation between the exponents 
$\widehat{w}_{n}$ and $w_n^{\ast}$.

\begin{thm} \label{hatandstar}
For every positive integer $n$ and every transcendental real number $\xi$, we have
$$
\widehat{w}_{n} (\xi) \le \frac{2 (w_n^{\ast} (\xi) + n) - 1}{3}
$$
and, if $w_{n} (\xi) \le 2n-1$, 
\begin{equation} \label{eq:fussball}
\widehat{w}_{n}^{\ast}(\xi)\geq 
\frac{2w_{n}^{\ast}(\xi)^{2}-w_{n}^{\ast}(\xi)-2n+1}
{2w_{n}^{\ast}(\xi)^{2}-nw_{n}^{\ast}(\xi)-n}.
\end{equation}
\end{thm}

It follows from the first assertion of Theorem \ref{hatandstar} that any counterexample $\xi$ 
to the Wirsing conjecture, that is, any transcendental real number 
$\xi$ with $w_n^{\ast} (\xi) < n$ for some integer $n \ge 3$, must satisfy  
$\widehat{w}_{n} (\xi) < \frac{4n - 1}{3}$. 

It follows from the second assertion of Theorem \ref{hatandstar} that if 
$w_n^{\ast} (\xi)$ is close to $\frac{n}{2}$ for some integer $n$
and some real transcendental number $\xi$, then 
$\widehat{w}_{n}^{\ast}(\xi)$ is also close to $\frac{n}{2}$. 
Note that \eqref{eq:sternhut} implies that \eqref{eq:fussball}  
holds for any couterexample $\xi$ to the Wirsing conjecture.

Theorem \ref{hatandstar} can be combined with \eqref{eq:wirrwarr} to get
a lower bound for $w_n^{\ast} (\xi)$ which is slightly smaller than the one 
obtained by Bernik and Tsishchanka \cite{bertish}. However, if we insert \eqref{eq:ssmj} 
in the proof of Theorem \ref{hatandstar}, then we get 
$$
{w}_{n}^{\ast} (\xi) \ge \max \Bigl\{ \widehat{w}_{n}(\xi), 
\frac{\widehat{w}_{n}(\xi)}{\widehat{w}_{n}(\xi)-n+1}, 
\frac{n-1}{2}\cdot \frac{\widehat{w}_{n}(\xi)^{2}-\widehat{w}_{n}(\xi)}
{1+(n-2)\widehat{w}_{n}(\xi)} + \widehat{w}_{n} (\xi) - n +\frac{1}{2}\Bigr\}.
$$
From this we derive a very slight improvement of \eqref{eq:bertis}, which,
like \eqref{eq:bertis}, has the form 
${w}_{n}^{\ast} (\xi) \ge \frac{n}{2} + 2 - \varepsilon_n$, where $\varepsilon_n$ is
positive and tends to $0$ when $n$ tends to infinity. Note that the best known
lower bound, established by Tsishchanka \cite{Tsi07}, 
has the form 
${w}_{n}^{\ast} (\xi) \ge \frac{n}{2} + 3 - \varepsilon'_n$, where $\varepsilon'_n$ is
positive and tends to $0$ when $n$ tends to infinity.

\section{Proofs}

We first show how Theorem~\ref{giltgleichheit} follows from 
Theorems~\ref{zendent} and~\ref{zeitung}. 

\begin{proof}[Proof of Theorem~\ref{giltgleichheit}] 

We distinguish two cases. 

If $w_{n-1}(\xi) = w_{n}(\xi)$,  then Theorem~\ref{zeitung} with $m=n-1$ implies that either
$$
\widehat{w}_{n}(\xi)\leq w_{n}(\xi)=w_{n-1}(\xi)\leq n-1+n-1 = 2n-2
$$
or
$$
\widehat{w}_{n}(\xi)\leq 2n-2.
$$
It then suffices to observe that $2n-2$ is smaller than the bounds 
in \eqref{eq:tomcat1} and \eqref{eq:tomcat2}.

If $w_{n-1}(\xi) < w_{n}(\xi)$, then we apply Theorem~\ref{zendent} with $m=n$ and we get
\[
\widehat{w}_{n} (\xi)\leq n+(n-1)\frac{\widehat{w}_{n}(\xi)}{w_{n}(\xi)},
\]
thus,
\begin{equation} \label{eq:borne1}
\widehat{w}_{n} (\xi)\leq  \frac{nw_{n}(\xi)}{w_{n}(\xi)-n+1}. 
\end{equation}
Rewriting inequality \eqref{eq:ssmj} as
\begin{equation} \label{eq:borne2}
\widehat{w}_{n}(\xi)\leq 
\frac{1}{2}\left(1+\frac{n-2}{n-1}w_{n}(\xi)\right)+
\sqrt{\frac{1}{4}\left(\frac{n-2}{n-1}w_{n}(\xi)+1\right)^{2}+\frac{w_{n}(\xi)}{n-1}}, 
\end{equation}
we have now two upper bounds for $\widehat{w}_{n}(\xi)$, one being given by 
a decreasing function and the other one by an increasing function of $w_{n}(\xi)$.     
An easy calculation shows that the
right hand sides of \eqref{eq:borne1} and \eqref{eq:borne2} are equal for
$$
w_{n}(\xi)= \frac{1}{2}\left(\frac{1 + 2n\sqrt{n^{2}-2n+\frac{5}{4}}}{n-1}+2n-1\right)
$$
Inserting this value in \eqref{eq:borne1} gives precisely the upper bound \eqref{eq:tomcat1}. 
For \eqref{eq:tomcat2} we proceed similarly 
using \eqref{eq:beesser} instead of \eqref{eq:ssmj}.
\end{proof}

For the proofs of Theorems~\ref{zendent} and~\ref{zeitung}
we need the following slight variation of~\cite[Lemma 8]{davsh}.

The notation $a \gg_d b$ means that $a$ exceeds $b$ times a constant
depending only on $d$. When $\gg$ is written without any subscript, it
means that the constant is absolute.

\begin{lem} \label{davschm}
Let $P, Q$ be coprime polynomials with integral coefficients 
of degrees at most $m$ and $n$, respectively. 
Let $\xi$ be a real number such that $\xi P(\xi) Q(\xi) \not= 0$. 
Then at least one of the two estimates 
\[
\vert P(\xi)\vert \gg_{m,n,\xi} H(P)^{-n+1}H(Q)^{-m}, 
\qquad \vert Q(\xi)\vert \gg_{m,n,\xi} H(P)^{-n}H(Q)^{-m+1}
\]
holds. In particular, we have 
\[
\max\{\vert P(\xi)\vert,\vert Q(\xi)\vert\} \gg_{m,n,\xi} H(P)^{-n+1}H(Q)^{-m+1}\min\{H(P)^{-1},H(Q)^{-1}\}.
\]
\end{lem}

\begin{proof}
We proceed as in the proof of \cite[Lemma~8]{davsh} 
and we consider the resultant $\rm{Res}(P,Q)$
of the polynomials $P$ and $Q$, written as 
\begin{align*}
P(T)&=a_{0}T^{s}+a_{1}T^{s-1}+\cdots+a_{s}, \qquad a_{0}\neq 0, s\leq m,     \\
Q(T)&=b_{0}T^{t}+b_{1}T^{t-1}+\cdots+b_{t}, \qquad \;\; b_{0}\neq 0, t\leq n. 
\end{align*}
Clearly, $\vert \rm{Res}(P,Q)\vert$ is at least $1$ since $P$ and $Q$ are coprime.
Transform the corresponding $(s+t)\times (s+t)$-matrix
by adding to the last column the sum, for $i=1, \ldots , s+t-1$, of the $(s+t-i)$-th column 
multiplied by $\xi^{i}$, so that the last column reads    
\[
(\xi^{t-1}P(\xi),\xi^{t-2}P(\xi),\ldots,P(\xi),\xi^{s-1}Q(\xi),\xi^{s-2}Q(\xi),\ldots,Q(\xi)).
\]
This transformation does not affect the value of $\rm{Res}(P,Q)$.
Observe that by expanding the determinant of the new matrix, we get that 
every product in the sum is in absolute value 
either $\ll_{s,t,\xi} \vert P(\xi)\vert H(P)^{t-1}H(Q)^{s}$
or $\ll_{s,t,\xi} \vert Q(\xi)\vert H(P)^{t}H(Q)^{s-1}$. 
Since there are only $(s+t)!\leq (m+n)!$ 
such terms in the sum we infer that 
$$
1\leq \vert {\rm{Res}} (P,Q)\vert 
\ll_{m,n,\xi} \max\{ \vert P(\xi)\vert H(P)^{n-1}H(Q)^{m},\vert Q(\xi)\vert H(P)^{n}H(Q)^{m-1}\}.
$$
The lemma follows.
\end{proof}

\begin{proof}[Proof of Theorem~\ref{zendent}] 

It is inspired from the proof of ~\cite[Proposition 2.1]{buglaur}. 
Let $m \ge n\geq 2$ be integers.
Let $\xi$ be a transcendental real number.
Assume first that $w_{m}(\xi)< + \infty$. 
We will show that if \eqref{eq:synode} is not satisfied, that is, 
if we assume
\begin{equation} \label{eq:waterfront}
w_{n-1}(\xi)<w_{m}(\xi),
\end{equation}
then \eqref{eq:denngiltsea} must hold.  
Let $\epsilon>0$ be an arbitrary but fixed small real number. By the definition of $w_{m}(\xi)$ 
there exist integer 
polynomials $P$ of degree at most $m$ and arbitrarily large height $H(P)$ such that
\begin{equation} \label{eq:guteglyx}
H(P)^{-w_{m}(\xi)-\epsilon} \leq \vert P(\xi)\vert\leq H(P)^{-w_{m}(\xi)+\epsilon}.
\end{equation}
By an argument of Wirsing~\cite[Hilfssatz~4]{wirsing} (see also on page 54 of \cite{bugbuch}),
we may assume that $P$ is irreducible.
We deduce from our assumption \eqref{eq:waterfront} that, if $\epsilon$ is small enough, 
then $P$ has degree at least $n$.
Moreover, by the definition of $\widehat{w}_{n}(\xi)$, 
if the height $H(P)$ is sufficiently large, then    
for all $X\geq H(P)$ the inequalities 
\begin{equation} \label{eq:auchguteglyx}
0 < \vert Q(\xi)\vert \leq X^{-\widehat{w}_{n}(\xi)+\epsilon}
\end{equation}
are satisfied by an integer polynomial $Q$ of degree at most $n$ and height $H(Q)\leq X$. 
Set $\tau(\xi,\epsilon)=(w_{m}(\xi)+2\epsilon)/(\widehat{w}_{n}(\xi)-\epsilon)$ and note
that this quantity exceeds $1$.
Keep in mind that 
\begin{equation} \label{eq:epskloayx}
\lim_{\epsilon\to 0} \tau(\xi,\epsilon)= \frac{w_{m}(\xi)}{\widehat{w}_{n}(\xi)}.
\end{equation}
For any integer polynomial $P$ satisfying \eqref{eq:guteglyx}, set
$X=H(P)^{\tau(\xi,\epsilon)}$.
Then \eqref{eq:guteglyx} implies
\begin{equation} \label{eq:epschen}
\vert P(\xi)\vert \geq H(P)^{-w_{m}(\xi)-\epsilon}> 
H(P)^{-w_{m}(\xi)-2\epsilon}=X^{-\widehat{w}_{n}(\xi)+\epsilon}, 
\end{equation}
thus any polynomial 
$Q$ satisfying \eqref{eq:auchguteglyx} also satisfies $\vert Q(\xi)\vert<\vert P(\xi)\vert$. 
Since $P$ is irreducible of degree at least $n$ and $Q$ has degree at most $n$, this 
implies that $P$ and $Q$ are coprime.

On the other hand, by \eqref{eq:guteglyx}, we have the estimate
\[
\vert P(\xi)\vert \leq H(P)^{-w_{m}(\xi)+\epsilon}=
X^{(-w_{m}(\xi)+\epsilon)/\tau(\xi,\epsilon)}.
\]
Thus, by \eqref{eq:epskloayx}, we get
\begin{equation} \label{eq:varepschen}
\vert P(\xi)\vert \leq X^{-\widehat{w}_{n}(\xi)+\epsilon^{\prime}}, 
\end{equation}
for some $\epsilon^{\prime}$ which depends on $\epsilon$ and tends to $0$ as 
$\epsilon$ tends to $0$. Since $\vert Q(\xi)\vert<\vert P(\xi)\vert$ 
we obviously obtain
\begin{equation} \label{eq:tobifroschx}
\max\{\vert P(\xi)\vert,\vert Q(\xi)\vert\}\leq 
X^{-\widehat{w}_{n}(\xi)+\epsilon^{\prime}}.
\end{equation}
We have constructed pairs of integer polynomials $(P, Q)$ 
of arbitrarily large height and satisfying \eqref{eq:tobifroschx}. 

We show that, provided $H(P)$ was chosen large enough, we have
\begin{equation} \label{eq:infer}
H(Q) \geq H(P)^{1-\epsilon^{\prime\prime}},
\end{equation}
where $\epsilon^{\prime\prime}$ is again some variation of $\epsilon$ and tends to $0$ as $\epsilon$ does.
Observe that since $\vert Q(\xi)\vert < \vert P(\xi)\vert$ and by \eqref{eq:guteglyx} we have
\[
w_{m}(\xi)-\epsilon\leq -\frac{\log \vert P(\xi)\vert}{\log H(P)} \leq -\frac{\log \vert Q(\xi)\vert}{\log H(P)}.
\]
On the other hand
\[
-\frac{\log \vert Q(\xi)\vert}{\log H(Q)} \leq
w_{n}(\xi)+\epsilon
\]
holds since $Q$ has degree at most $n$ and can be considered of
sufficiently large height $H(Q)$. Moreover the assumption $m\geq n$ implies
$w_{m}(\xi)\geq w_{n}(\xi)$. Combination of these facts yields 
\[
\frac{\log H(Q)}{\log H(P)} = \left(-\frac{\log \vert Q(\xi)\vert}{\log H(P)}\right)
\cdot \left(-\frac{\log \vert Q(\xi)\vert}{\log H(Q)}\right)^{-1}
\geq \frac{w_{m}(\xi)-\epsilon}{w_{n}(\xi)+\epsilon}\geq \frac{w_{n}(\xi)-\epsilon}{w_{n}(\xi)+\epsilon},
\]
and we indeed infer \eqref{eq:infer} as $\epsilon$ tends to $0$.

Now observe that we can apply
Lemma~\ref{davschm} to the coprime polynomials $P$ and $Q$.
In case of $H(Q)\geq H(P)$ for infinitely many such pairs $(P,Q)$ we get
\begin{equation} \label{eq:liostadt2yx}
\max\{\vert P(\xi)\vert,\vert Q(\xi)\vert\} \gg_{m,n,\xi}
H(P)^{-n+1}H(Q)^{-m}\geq X^{-\frac{n-1}{\tau(\xi,\epsilon)}}  X^{-m}.
\end{equation}
Combining \eqref{eq:tobifroschx} and \eqref{eq:liostadt2yx} 
we deduce \eqref{eq:denngiltsea} as $\epsilon$ can be taken arbitrarily small. 
If otherwise $H(Q)<H(P)$ for all large pairs $(P,Q)$, Lemma~\ref{davschm} yields
\[
\max\{\vert P(\xi)\vert,\vert Q(\xi)\vert\} \gg_{m,n,\xi}
H(P)^{-n+1}H(Q)^{-m+1}\cdot H(P)^{-1}, 
\]
however since $H(Q)$ cannot be much smaller than $H(P)$
by \eqref{eq:infer} we similarly infer
\begin{align} \label{eq:liostadt2yxr}
\max\{\vert P(\xi)\vert,\vert Q(\xi)\vert\} 
\geq X^{-\frac{n-1}{\tau(\xi,\epsilon)}}  X^{-m+\epsilon^{\prime\prime\prime}},
\end{align}
where $\epsilon^{\prime\prime\prime}=1/(1-\epsilon^{\prime\prime})-1$ again tends to $0$ as $\epsilon$ does.
The claim follows again with $\epsilon$ to $0$ and we have completed
the proof of the case $w_{m}(\xi)<+\infty$.

If $w_{m}(\xi) = + \infty$, we take a sequence $(P_j)_{j \ge 1}$ of integer 
polynomials of degree at most $m$ with increasing heights and such that the quantity
$- \log |P_j (\xi)| / \log H(P_j)$ tends to infinity as $j$ tends to infinity. We proceed 
then exactly as above, by using this sequence of polynomials instead of 
the polynomials satisfying \eqref{eq:guteglyx}. We omit the details. 
\end{proof}

\begin{proof}[Proof of Theorem~\ref{zeitung}] 

We assume $n\geq 2$ and $w_{m}(\xi)< + \infty$, for similar reasons as 
in the previous proof. 
Let $\epsilon>0$ be an arbitrary but fixed small number. By the definition of $w_{m}(\xi)$, 
there exist integer
polynomials $P$ of degree at most $m$ and arbitrarily large height $H(P)$ such that
$$
\vert P(\xi)\vert\leq H(P)^{-w_{m}(\xi)+\epsilon/2}.
$$
Again, by using an argument of Wirsing \cite[Hilfssatz 4]{wirsing}, we can assume that $P$
is irreducible. Then, by \cite[Lemma A.3]{bugbuch}, there exists a 
real number $K(n)$ in ${(0,1)}$ such 
that no integer polynomial $Q$ of degree at most $n$ and whose
height satisfies $H(Q)\leq K(n)H(P)$ is a multiple of $P$. 
Set $X:= H(P)K(n)/2$. If $X$ is large enough, then the polynomial $P$ satisfies 
\begin{equation} \label{eq:jungle}
\vert P(\xi)\vert\leq X^{-w_{m}(\xi)+\epsilon}.
\end{equation}
On the other hand, by the definition of $\widehat{w}_{n}(\xi)$, 
we may consider only the polynomials 
$P$ for which $H(P)$ is sufficiently large, so that the estimate
\begin{equation} \label{eq:auchguteglyv}
0 < \vert Q(\xi)\vert \leq X^{-\widehat{w}_{n}(\xi)+\epsilon}
\end{equation}
holds for an integer polynomial $Q$ of degree at most $n$ and height $H(Q)\leq X$. 
Our choice of $X$ ensures that $Q$ is not a multiple of $P$. 
Since $P$ is irreducible, the polynomials
$P$ and $Q$ are coprime. Thus we may apply Lemma~\ref{davschm} which yields
\begin{equation} \label{eq:frisch}
\max\{\vert P(\xi)\vert,\vert Q(\xi)\vert\} \gg_{m,n,\xi} X^{-m-n+1}.
\end{equation}
Combining \eqref{eq:jungle}, \eqref{eq:auchguteglyv} and \eqref{eq:frisch}, we deduce that 
$\min\{w_{m}(\xi),\widehat{w}_{n}(\xi)\}\leq m+n-1$,
as $\epsilon$ can be taken arbitrarily small.
\end{proof}

\begin{proof}[Proof of Theorem~\ref{simulterne}]
Most estimates arise by a modification of the proof of Theorem~\ref{zendent}.
Define the irreducible polynomial $P$ as in the proof of Theorem~\ref{zendent}.
In that proof   
a difficulty occurs since the polynomial $Q$
which satisfies \eqref{eq:auchguteglyx} is not a priori coprime with $P$. 
The assumption \eqref{eq:waterfront} was used to guarantee that $Q$ is not a multiple of $P$. 

Here, instead of \eqref{eq:auchguteglyx}, we use the fact that,  
for all $X\geq H(P)$, the inequalities 
\begin{equation} \label{eq:neunund}
0 < |\xi - \beta| < H(\beta)^{-1} \, X^{-\widehat{w}_{n}^{\ast} (\xi) + \epsilon}
\end{equation}
are satisfied by an algebraic number $\beta$ of degree at most $n$ and height 
at most $X$. Let $Q$ be the minimal defining polynomial over ${\mathbb Z}$   
of such a $\beta$. Then a standard argument yields
\begin{equation} \label{eq:a1}
\vert Q(\xi)\vert \ll_{n,\xi} X^{-\widehat{w}_{n}^{\ast} (\xi) + \epsilon},
\end{equation}
see~\cite[Proposition~3.2]{bugbuch} 
(actually, this proves the left inequalities of \eqref{eq:sternhut}).
Next we define 
$\tau^{\ast}(\xi,\epsilon):=(w_{m}(\xi)+2\epsilon)/(\widehat{w}_{n}^{\ast}(\xi)-\epsilon)$
and set $X=H(P)^{\tau^{\ast}(\xi,\epsilon)}$. Similarly as in 
the proof of Theorem~\ref{zendent} 
we obtain the variant 
\begin{equation} \label{eq:lastlast}
\vert P(\xi)\vert \geq H(P)^{-w_{m}(\xi)-\epsilon}= 
H(P)^{\epsilon}H(P)^{-w_{m}(\xi)-2\epsilon}=H(P)^{\epsilon}X^{-\widehat{w}_{n}^{\ast}(\xi)+\epsilon}
\end{equation}
of \eqref{eq:epschen}. 
Observe that the combination of \eqref{eq:a1} and \eqref{eq:lastlast} 
implies that $\vert Q(\xi)\vert<\vert P(\xi)\vert$ and consequently 
$P\neq Q$, provided that $H(P)$ was chosen large enough.
On the other hand, with essentially the argument 
used to get \eqref{eq:varepschen}, we obtain 
\begin{equation} \label{eq:a2}
\vert P(\xi)\vert \leq X^{-\widehat{w}_{n}^{\ast}(\xi)+\tilde{\epsilon}}, 
\end{equation}
for some $\tilde{\epsilon}$ which depends on $\epsilon$ and tends to $0$ as 
$\epsilon$ tends to $0$.   
By \eqref{eq:a1} and $\vert Q(\xi)\vert<\vert P(\xi)\vert$ we infer
\begin{equation} \label{eq:tobifroschxxx}
\max\{\vert P(\xi)\vert,\vert Q(\xi)\vert\}\leq 
X^{-\widehat{w}_{n}^{\ast}(\xi)+\tilde{\epsilon}},
\end{equation}
an inequality similar to \eqref{eq:tobifroschx}. 

Now if $m\geq n$, we proceed as in the proof of Theorem~\ref{zendent} 
observing that we may apply Lemma~\ref{davschm} here since $P\neq Q$ 
and both $P$ and $Q$ are irreducible.
Indeed \eqref{eq:infer} holds for exactly the same reason
and distinguishing the cases $H(P)\leq H(Q)$ and $H(P)>H(Q)$
again gives \eqref{eq:liostadt2yx} and \eqref{eq:liostadt2yxr} respectively with 
$\tau$ replaced by $\tau^{\ast}$. 
%
%
This yields the left inequality of \eqref{eq:value}
whereas the right inequality $\widehat{w}_{n}^{\ast}(\xi)\leq w_{m}(\xi)$ is 
trivially implied in case of $m\geq n$.

If $m<n$ we treat the cases $H(P)\leq H(Q)$ and $H(P)>H(Q)$ separately. 
First consider the case $H(P)\leq H(Q)$ for infinitely many pairs $(P,Q)$ as above. 
In this case we again prove \eqref{eq:value}. The 
left inequality of \eqref{eq:value} is derived precisely as in the case $m\geq n$, as we did not
utilize \eqref{eq:infer} for the proof.
%
%
The other inequality $\widehat{w}_{n}^{\ast}(\xi)\leq w_{m}(\xi)$ remains to be shown. 
Assume otherwise $w_{m}(\xi)/\widehat{w}_{n}^{\ast}(\xi)<1$.
Then for sufficiently small $\epsilon$ also $\tau^{\ast}(\xi,\epsilon)<1$ and 
hence $H(Q)=H(\beta)\leq X < X^{1/\tau^{\ast}(\xi,\epsilon)}=H(P)$, contradiction. 
The proof of the case $H(P)\leq H(Q)$ is finished.

Now assume $H(P)>H(Q)$ for infinitely many pairs $(P,Q)$ as above. Note 
that \eqref{eq:infer} does not necessarily hold now as we needed $m\geq n$ for its deduction.
In this case we show that \eqref{eq:negate} is false, that is
\begin{equation} \label{eq:froh}
w_{m}(\xi)\leq \min\{ m+n-1,w_{n}^{\ast}(\xi)\}.
\end{equation}
Provided \eqref{eq:froh} holds one readily checks the logical implication of the theorem. 
Observe \eqref{eq:tobifroschxxx} implies
\begin{equation} \label{eq:sonne}
\max\{ \vert P(\xi)\vert, \vert Q(\xi)\vert \} \leq H(P)^{-w_{m}(\xi)+\hat{\epsilon}}
\end{equation}
for $\hat{\epsilon}=\tilde{\epsilon}\cdot w_{m}(\zeta)/\widehat{w}_{n}^{\ast}(\zeta)$ 
which again tends to $0$ as $\epsilon$ does. On the other
hand Lemma~\ref{davschm} yields
\begin{equation} \label{eq:mond}
\max\{ \vert P(\xi)\vert, \vert Q(\xi)\vert \} \gg_{m,n,\xi} H(P)^{-n}H(Q)^{-m+1}\geq H(P)^{-m-n+1}.
\end{equation}
Combination of \eqref{eq:sonne} and \eqref{eq:mond} yields $w_{m}(\xi)\leq m+n-1$.
It remains to be be shown that $w_{m}(\xi)\leq w_{n}^{\ast}(\xi)$. Assume this is false
and we have $w_{m}(\xi)-w_{n}^{\ast}(\xi)=\rho>0$. Then \eqref{eq:neunund} would imply,
if $\epsilon$ was chosen small enough, that
\begin{align*}
\vert \xi-\beta\vert &\leq H(\beta)^{-1} \, X^{-\widehat{w}_{n}^{\ast} (\xi) + \epsilon}
=H(Q)^{-1}H(P)^{-w_{m}(\xi)+\epsilon/\tau(\xi,\epsilon)}  \\
&< H(Q)^{-w_{m}(\xi)-1+\epsilon/\tau(\xi,\epsilon)}\leq H(Q)^{-w_{n}^{\ast}(\xi)-1-\rho/2},
\end{align*}
contradiction to the definition of $w_{n}^{\ast}(\xi)$ as $H(Q)$ tends to infinity. 
Hence \eqref{eq:froh} is established in this case and the proof is finished.
\end{proof}

\begin{proof}[Proof of Theorem~\ref{hatandstar}]

Let $n \ge 2$ be an integer and $\xi$ be a real transcendental number. 

We establish the first assertion. 
We follow the proof of Wirsing's theorem as given in \cite{bugbuch} and keep 
the notation used therein.
By the definition of $\widehat{w}_{n}$, observe that the inequality
$|Q_k (\xi)| \ll H (P_k)^{-n}$ in \cite[(3.16)]{bugbuch}
can be replaced by 
$$
|Q_k (\xi)| \ll H (P_k)^{-\widehat{w}_{n} (\xi) + \varepsilon}
$$
The lower bound for ${w}_{n}^{\ast} (\xi)$ on line $-8$ of \cite[p. 57]{bugbuch} then becomes
\begin{equation}\label{eq:longue}
{w}_{n}^{\ast} (\xi) \ge \min \Bigl\{ \widehat{w}_{n} (\xi),  
w_n (\xi) - \frac{n-1}{2}  + \frac{\widehat{w}_{n} (\xi) - n}{2},
\frac{w_n(\xi) + 1}{2} + \widehat{w}_{n} (\xi) - n \Bigr\}. 
\end{equation}
Since $w_n (\xi) \ge \widehat{w}_{n} (\xi)$, this gives
$$
{w}_{n}^{\ast} (\xi)  \ge \min \Bigl\{ \widehat{w}_{n} (\xi), 
\frac{3 \widehat{w}_{n} (\xi)}{2} - n + \frac{1}{2} \Bigr\} 
= \frac{3 \widehat{w}_{n} (\xi)}{2} - n + \frac{1}{2}, 
$$
by \eqref{eq:glmschr}. Thus, we have established
$$
\widehat{w}_{n} (\xi) \le \frac{2 (w_n^{\ast} (\xi) + n) - 1}{3}, 
$$
as asserted.

Now, we prove \eqref{eq:fussball}.
Inequality \eqref{eq:wirrwarr} can be rewritten as
$$
\widehat{w}_{n} (\xi) \ge \frac{(n-1) {w}_{n}^{\ast} (\xi)}{{w}_{n}^{\ast} (\xi)-1}.
$$
Assuming ${w}_{n} (\xi) \le 2n-1$, the smallest of the three terms in the curly brackets in 
\eqref{eq:longue} is the third one and we eventually get
$$
w_n (\xi) \le \frac{2 w_n^{\ast} (\xi)^2 - 2 n - w_n^{\ast} (\xi) + 1}{w_n^{\ast} (\xi)  - 1}.
$$
Combined with the lower bound
$$
\widehat{w}_{n}^{\ast}(\xi)\geq \frac{w_{n}(\xi)}{w_{n}(\xi)-n+1},
$$
established in \cite{buglaur}, we obtain \eqref{eq:fussball}. 
\end{proof}

\subsection*{Acknowledgements}
The authors are very grateful to the referee for a careful reading and several useful 
suggestions.
Johannes Schleischitz is supported by FWF grant P24828.

\end{document}